\documentclass[12pt,leqno]{article}
\usepackage[shortlabels]{enumitem}
\usepackage[utf8]{inputenc} 
\usepackage[T1]{fontenc}
\usepackage{preamble}

\title{Mean curvature of direct image bundles 
}
\author{Kuang-Ru Wu}

\usepackage{fullpage}
\usepackage{multicol}

\begin{document}

\date{}

\parskip=6pt

\maketitle

\begin{abstract}

Let $E\to X$ be a vector bundle of rank $r$ over a compact complex manifold $X$ of dimension $n$. It is known that if the line bundle $O_{P(E^*)}(1)$ over the projectivized bundle $P(E^*)$ is positive, then $E\otimes \det E$ is Nakano positive by the work of Berndtsson. In this paper, we give a subharmonic analogue. Let $p:P(E^*)\to X$ be the projection and $\alpha$ be a K\"ahler form on $X$. If the line bundle $O_{P(E^*)}(1)$ admits a metric $h$ with curvature $\Theta$ positive on every fiber and $\Theta^r\wedge p^*\alpha^{n-1}> 0$, then $E\otimes \det E$ carries a Hermitian metric whose mean curvature is positive. 

As an application, we show that the following subharmonic analogue of the Griffiths conjecture is true: if the line bundle $O_{P(E^*)}(1)$ admits a metric $h$ with curvature $\Theta$ positive on every fiber and $\Theta^r\wedge p^*\alpha^{n-1}> 0$, then $E$ carries a Hermitian metric with positive mean curvature.

%It is known that if a vector bundle $E$ is ample, then $E\otimes \det E$ is strictly Nakano positive by the work of Berndtsson. An equivalent statement using Finsler metrics is that if $E^*$ is Kobayashi negative, then $E\otimes \det E$ is strictly Nakano positive. In this paper, we prove a subharmonic analogue. If $E^*$ is Kobayashi mean negative, then $E\otimes \det E$ carries a Hermitian metric whose mean curvature is positive. We also prove the dual statement. If $E^*$ is Kobayashi mean positive, then $E\otimes \det E$ carries a Hermitian metric with negative mean curvature. 
\end{abstract}

\section{Introduction}

The main goal of the paper is to establish some positivity results about the mean curvature of the direct image bundles for holomorphic fibrations and use these results to prove the subharmonic analogue of the Griffiths conjecture. Let us begin with the case of trivial fibrations. Consider a product $D\times X$ where $D$ is a bounded open set in $\mathbb{C}^m$ and $X$ is a compact complex manifold of dimension $n$. Let $(L,h)$ be a Hermitian line bundle over $X$ with positive curvature $\Theta(h):=\omega>0$. For a function $\phi:D\times X \to \mathbb{R}$, we define a Hermitian metric $H$ on the trivial bundle $V:=D\times H^0(X,L\otimes K_X)\to D$ as follows. For $s\in V_t$ with $t\in D$,  
\begin{equation}\label{metric 1}
  H(s,s):=\int_X h(s,s)e^{-\phi(t,\cdot)}.  
\end{equation}
Here we extend the metric $h$ to acting on sections $s$ of $L\otimes K_X$ so that $h(s,s)$ is an $(n,n)$-form on $X$. In terms of local coordinates in $X$, if $s=\sigma\otimes e$ with $\sigma$ an $(n,0)$-form and $e$ a frame of $L$, then $h(s,s)=c_n\sigma \wedge \bar{\sigma} h(e,e)$ where $c_n=i^{n^2}$.

When the function $\phi$ satisfies $\pi_2^*\omega+i\partial\Bar{\partial} \phi>0$ where $\pi_2:D\times X \to X$ is the projection, the Hermitian vector bundle $(V,H)$ is Nakano positive by the work of Berndtsson \cite{Berndtsson09}. In order to state our results, let us fix a positive Hermitian form $\alpha$ on $D$ and denote by $\pi_1$ the projection $D\times X \to D$. We have the following subharmonic analogue of Berndtsson's theorem.

\begin{theorem}\label{thm 1}
    If $\phi\in C^2(D\times X)$, $\omega+i\partial\Bar{\partial}\phi(t,\cdot)>0$ for any $t\in D$, and $(\pi_2^*\omega+i\partial\Bar{\partial} \phi)^{n+1} \wedge \pi_1^*\alpha^{m-1}> 0$, then the bundle $(V,H)$ has positive mean curvature $\Lambda_\alpha\Theta^V> 0$. 
The semipositive case is also true: if $(\pi_2^*\omega+i\partial\Bar{\partial} \phi)^{n+1} \wedge \pi_1^*\alpha^{m-1}\geq 0$, then $\Lambda_\alpha\Theta^V\geq  0$.
\end{theorem}

When $m=1$, Theorem \ref{thm 1} is the same as Berndtsson's theorem. The function $\phi$ in Theorem \ref{thm 1} is a special case of subharmonicity on graphs introduced in \cite{wu23}. A more general theorem that allows for such $\phi$ will be stated in Subsection \ref{subsection 2.1}. The concept of subharmonicity on graphs and a special case of Theorem \ref{thm 1} are used in \cite{wu23} to study the approximation/quantization of the Wess--Zumino--Witten (WZW) equations in the space of K\"ahler potentials through the Hermitian--Yang--Mills equations. The WZW equation in question is $(\pi_2^*\omega+i\partial\Bar{\partial} \phi)^{n+1} \wedge \pi_1^*\alpha^{m-1}=0 $. Therefore, the weight function $\phi$ is actually a subsolution of the WZW equation. For motivation and recent developments for the WZW equation, see \cite{Semmesmonge, Donaldson99, RubinsteinZelditch, LempertLegendre,wu23,wu2024potential,finski2024WZW,finski2024lower}. 

\subsection{Nontrivial fibrations}
For the case of nontrivial fibrations, we consider a proper holomorphic map $p:\mathcal{X}^{n+m}\to Y^m$ between two complex manifolds with $\mathcal{X}$ being K\"ahler and the differential being surjective at every point. We denote the fibers $p^{-1}(t)$ by $\mathcal{X}_t$ for $t\in Y$. Let $(L,h)$ be a Hermitian line bundle over $\mathcal{X}$ and let $$V_t=H^0(\mathcal{X}_t, L|_{\mathcal{X}_t}\otimes K_{\mathcal{X}_t}).$$ We assume that $\dim V_t$ is independent of $t\in Y$ (for our application later, this assumption is satisfied). So, the direct image of the sheaf of sections of $L\otimes K_{\mathcal{X}/Y}$ is locally free by Grauert's direct image theorem. We denote by $V$ the associated vector bundle over $Y$. There is a naturally defined Hermitian metric $H$ on $V$. For $u$ in $V_t$ with $t\in Y$, 
\begin{equation}\label{metric}
  H(u,u):=\int_{\mathcal{X}_t}h(u,u).  
\end{equation}
As in the trivial fibration, we extend the metric $h$ to acting on sections $u$ of $L|_{\mathcal{X}_t}\otimes K_{\mathcal{X}_t}$ so that $h(u,u)$ is an $(n,n)$-form on $\mathcal{X}_t$. In terms of local coordinates, if $u=u'\otimes e$ with $u'$ an $(n,0)$-form and $e$ a frame of $L|_{\mathcal{X}_t}$, then $h(u,u)=c_n u' \wedge \overline{u'} h(e,e)$ where $c_n=i^{n^2}$.

When the metric $h$ has positive curvature, the bundle $(V, H)$ is Nakano positive by the work of Berndtsson \cite{Berndtsson09}. If we let $\alpha$ be a positive Hermitian form on $Y$, then we have the following subharmonic analogue which is the same as Berndtsson's theorem when $m=1$.

\begin{theorem}\label{thm 4}
    If the curvature $\Theta$ of $h$ is positive on every fiber and $\Theta^{n+1}\wedge p^*\alpha^{m-1}>0 $, then $(V, H)$ has positive mean curvature $\Lambda_\alpha \Theta^V> 0$. 
The semipositive case is also true: if $\Theta^{n+1}\wedge p^*\alpha^{m-1}\geq 0$, then $\Lambda_\alpha \Theta^V\geq  0$.
\end{theorem}

As in \cite{BerndtssonPaun}, one can define the relative Bergman kernel metric $B$ on the line bundle $L\otimes K_{\mathcal{X}/Y}$ where $K_{\mathcal{X}/Y}$ is the relative canonical bundle (the relevent definitions will be recalled in Subsection \ref{sub 3.1}). As a consequence of Theorem \ref{thm 4}, the curvature of $B$ has the same  positivity as the curvature of $h$: 
\begin{corollary}\label{cor relative} 
 If the curvature $\Theta$ of $h$ is positive on every fiber and $\Theta^{n+1}\wedge p^*\alpha^{m-1}\geq 0$, then the curvature $\Theta(B)$ of the relative Bergman kernel metric $B$ on $L\otimes K_{\mathcal{X}/Y}$ satisfies $\Theta(B)^{n+1}\wedge p^*\alpha^{m-1}\geq 0$.
\end{corollary}

\subsection{Projectivized bundles}

One example of the nontrivial fibration is $p:P(E^*)\to X$ where $E$ is a holomorphic vector bundle of rank $r$ over a compact complex manifold $X$ of dimension $n$ and $P(E^*)$ is the projectivized bundle of $E^*$. Let $h$ be a metric on $O_{P(E^*)}(1)$ and $\alpha$ be a K\"ahler form on $X$. As a corollary of Theorem \ref{thm 4}, we have
\begin{theorem}\label{thm EdetE}
     Assume the curvature $\Theta$ of $h$ is positive on every fiber. 
     \begin{enumerate}
         \item If $\Theta^{r}\wedge p^*\alpha^{n-1}\geq 0 
 (\text{or} >0)$, then $(S^k E\otimes \det E, H)$ has semipositive (or positive) mean curvature for any $k\geq 0$. 
         \item 
         If $\Theta^{r}\wedge p^*\alpha^{n-1}> 0$, then 
         $(S^k E,H)$ has positive mean curvature for $k$ large.
      \end{enumerate}
\end{theorem}
The vector bundle $V$ in Theorem \ref{thm 4} is associated with the direct image of $L\otimes K_{\mathcal{X}/Y}$, therefore the first statement in Theorem \ref{thm EdetE} follows when we  use $O_{P(E^*)}(r+k)$ for the line bundle $L$ and the fact that the relative canonical bundle
$K_{P(E^*)/X}$ is isomorphic to $O_{P(E^*)}(-r)\otimes p^*\det E$. For the second statement, we use $O_{P(E^*)}(k)\otimes K^{-1}_{P(E^*)/X}$ for $L$ and equip $K^{-1}_{P(E^*)/X}$ with some arbitrary metric. In order to use Theorem \ref{thm 4}, we have to make sure $P(E^*)$ is K\"ahler in the present setup. This fact is perhaps known to the experts, but we still provide a proof in Subsection \ref{subsection 5.2} (see also \cite[Theorem 8]{Kodaira}).

In \cite{Griff69}, Griffiths conjectured that if $O_{P(E^*)}(1)$ is positive, then $E$ admits a Hermitian metric with Griffiths positive curvature. The conjecture is still open, and the closest result so far is perhaps the one due to Berndtsson \cite{Berndtsson09} where he shows that $E\otimes \det E$ is Nakano positive. For the progress on the Griffiths conjecture, see  \cite{Umemura,CampanaFlenner,Berndtsson09,MourouganeTaka,toric,positivityandvanishingthmliu,naumann2017approach,FengLiuWan,demailly2020hermitianyangmills,pingali2021note,Finskichara, wu_2022,wupositivelyII,wuIII,Mazhang,lempert2024two}. Let us formulate the relation between ampleness and positivity a bit further:
\begin{enumerate}
    \item\label{1} $E$ admits a Hermitian metric with
Griffiths positive curvature.    \item\label{2} $O_{P(E^*)}(1)$ is positive.
    \item\label{3}  $O_{P(E)}(1)$ admits a Hermitian metric whose curvature has signature $(r-1,n)$ and is positive on every fiber.
\end{enumerate}
Statement \ref{1} implies both \ref{2} and \ref{3} by a standard computation which will be recalled later. Statement \ref{3} is equivalent to $E$ being Kobayashi positive, and it implies statement \ref{2} (see \cite{FengLiuWan,wu_2022}). The direction from \ref{2} to \ref{1} and the one from \ref{2} to \ref{3} are respectively the Griffiths conjecture and the Kobayashi conjecture. Following Griffiths and Kobayashi, we ask whether the statements below are equivalent: \begin{enumerate}[label=\Alph*.]

\item\label{a} $E$ admits a Hermitian metric with positive mean curvature.

         \item\label{b} $O_{P(E^*)}(1)$ admits a Hermitian metric whose curvature $\Theta$ is positive on every fiber and $\Theta^{r}\wedge p^*\alpha^{n-1}> 0$.

      \item\label{c}
      $O_{P(E)}(1)$ admits a Hermitian metric whose curvature $\Theta$ is positive on every fiber and $\Theta^{r}\wedge q^*\alpha^{n-1}<  0$, where $q:P(E)\to X$ is the projection.
      \end{enumerate}
We remark that when the dimension of $X$ is one, statements \ref{1}, \ref{2}, and \ref{3} are statements A, B, and C respectively, and in this case, all the statements are equivalent by \cite{Umemura,CampanaFlenner}. 

As we indicated earlier, Berndtsson's positivity theorem is not enough to prove the Griffiths conjecture. However, somewhat surprisingly, the subharmonic analogue of the Griffiths conjecture (namely, the equivalence of A, B, and C)  can be proved using Theorem \ref{thm EdetE} combined with recent developments in Hermitian--Yang--Mills metrics. This is one of the motivations of establishing Theorem \ref{thm EdetE}. Indeed, we have

\begin{theorem}\label{thm 5}
  Statements A, B, and C are equivalent.    
\end{theorem}

The fact that statement A implies both B and C follows by a standard computation and will be shown in Section \ref{The subharmonic analogue of the Griffiths conjecture}. For the implication from B to A, we will provide three proofs. The first one relies on Theorem \ref{thm EdetE} and results in \cite{li2021mean}, whereas the second one uses the asymptotic expansion of the Bergman kernel and the fibered Yang--Mills functional in \cite{Mazhang,finski2024lower}. The third one uses Theorem \ref{thm EdetE} and the classical results in the existence of Hermitian--Yang--Mills metrics on stable bundles \cite{reprintDiffofcomplexbundles,uhlenbeckyau,AdamJacob}, but at this point, the third proof only works when the bundle $E$ is of rank 2, or the dimension of $X$ is 2. For the implication from C to A, we use again the results in \cite{Mazhang,finski2024lower}.

Before ending the introduction, let us say a few words about the proofs of Theorems \ref{thm 1} and \ref{thm 4}. As mentioned already, our theorems are subharmonic analogue of Berndtsson's theorems, hence our approach follows closely Berndtsson's computations in \cite{Berndtsson09}. The novelty in our results is perhaps discovering the suitable assumption for the weight function - the subharmonicity on graphs (equivalently subsolutions of the WZW equation) under which the direct image bundles have positive mean curvature (equivalently subsolutions of the HYM equation). Another noticeable difference from Berndtsson's computations is when we estimate the middle term in formula (\ref{4.5'}) using subharmonicity on graphs which is more complicated than the case of plurisubharmonicity.

The paper is organized as follows. In Section \ref{sec 2}, we prove Theorem \ref{thm 1} and its generalization. In Section \ref{sec 3}, we prove Theorem \ref{thm 4} and Corollary \ref{cor relative}. In Section \ref{The subharmonic analogue of the Griffiths conjecture}, we discuss the subharmonic analogue of the Griffiths conjecture and prove Theorem \ref{thm 5} for some special cases in Subsection \ref{subsec special} an the general case in Subsection \ref{subsec general}. 

I am grateful to L\'aszl\'o Lempert for his remarks on the draft of the paper. I would like to thank Siarhei Finski for pointing out that the results in \cite{finski2024lower} can be used to prove Theorem \ref{thm 5} in full generality.
Thanks are also due to National Science and Technology Council, National Tsing Hua University, and the Erd\H{o}s center for their support.

\section{Proof of Theorem \ref{thm 1}}\label{sec 2}
Recall the setup for the case of trivial fibrations $\pi_1: D\times X \to D$ where $D$ is a bounded open set in $\mathbb{C}^m$ with a positive Hermitian form $\alpha$ and $X$ is a compact complex manifold of dimension $n$. We assume there is a Hermitian line bundle $(L,h)$ over $X$ with positive curvature $\Theta:=\omega>0$. We define a Hermitian metric $H$ on $V:=D\times H^0(X,L\otimes K_X)\to D$ as in (\ref{metric 1}) with a weight function $\phi$ on $D\times X$. 

We start with a formula for any $\phi\in C^2(D\times X)$:
\begin{equation}\label{2.1}
\begin{aligned}
  &(\pi_2^*\omega+i\partial \bar{\partial}\phi)^{n+1}\wedge(\pi_1^*\alpha)^{m-1}\\
  =&(n+1)! (m-1)!
 \sum_{i,j} \alpha^{i\bar{j}}\det M(i\bar{j})
   \det(\alpha)
  \big(\bigwedge^m_{k=1} i dt_k\wedge d\Bar{t}_k\wedge \bigwedge^n_{\lambda=1} i dz_\lambda\wedge d\Bar{z}_\lambda\big). 
\end{aligned}
\end{equation}
Here $M(i\bar{j})$ for fixed $i$ and $j$ is a matrix defined by
\begin{equation}\label{matrix}
M(i\bar{j}):=\left (
\begin{array}{cccc}
(\phi+\psi)_{t_i\bar{t}_j}
& (\phi+\psi)_{t_i\bar{z}_1} & \cdots &  (\phi+\psi)_{t_i\Bar{z}_n}\\
(\phi+\psi)_{z_1\bar{t}_j} & (\phi+\psi)_{z_1\bar{z}_1} & \cdots & (\phi+\psi)_{z_1\bar{z}_n}\\
 \vdots & \vdots & \ddots & \vdots \\
(\phi+\psi)_{z_n\bar{t}_j}& (\phi+\psi)_{z_n\bar{z}_1} &\cdots & (\phi+\psi)_{z_n\bar{z}_n} 
\end{array}
\right )
\end{equation}
with respect to coordinates $\{t_j\}$ in $D$ and local coordinates $\{z_\lambda\}$ in $X$ where $\psi$ a local potential of $\omega$ is defined. Formula (\ref{2.1}) is a generalization of formula (15) in \cite{wu23}. To prove formula (\ref{2.1}), one can first apply a linear transformation of constant coefficients such that the metric $\alpha$ is Euclidean at one point, and then apply  formula (15) in \cite{wu23} (the notation for coordinates here is slightly different from \cite{wu23}).

\begin{proof}[Proof of Theorem \ref{thm 1}]

Let $L^2(X,L\otimes K_X)$ be the space of measurable sections $s$ whose $L^2$ norm $\int_X h(s,s)e^{-\phi(t,\cdot)}$ is finite. Since different $t$ will give rise to comparable $L^2$ norms, the space $L^2(X,L\otimes K_X)$ does not change with $t$, and so we have a Hermitian Hilbert bundle $F:=D\times L^2(X,L\otimes K_X)\to D$ which has $V=D\times H^0(X,L\otimes K_X)\to D$ as a subbundle. This setup is almost identical to \cite[Theorem 1.1]{Berndtsson09}.

Denote the curvature of the bundles $V$ and $F$ by $\Theta^V=\sum \Theta^V_{j\bar{k}}dt_j\wedge d\bar{t}_k$ and $\Theta^F=\sum \Theta^F_{j\bar{k}}dt_j\wedge d\bar{t}_k$
respectively. Following the computation in \cite[Section 3]{Berndtsson09}, we have $\Theta^F_{j\bar{k}}=\partial^2\phi/\partial t_j\partial\bar{t}_k=\phi_{j\bar{k}}$ and the second-fundamental-form formula (after taking trace $\Lambda_\alpha$)
\begin{equation}\label{2nd form}
    \sum_{i,j} \alpha^{i\bar{j}}(\Theta^V_{i\Bar{j}}s,s)=\sum_{i,j} \alpha^{i\bar{j}}(\phi_{i\Bar{j}}s,s)-\sum_{i,j} \alpha^{i\bar{j}}(\pi_{\perp} D^F_{t_i}s,\pi_{\perp} D^F_{t_j}s)
\end{equation}
where $D^F_{t_j}$ is the $(1,0)$-part of the Chern connection on $F$, $s$ is a local smooth section of $V$, $(\alpha^{i\bar{j}})$ is the inverse matrix of $(\alpha_{i\bar{j}})$ for $\alpha$, and $\pi_{\perp}$ is the projection on the orthogonal complement of $V$ in $F$. 

In order to show that $\Lambda_\alpha \Theta^V> 0$, we only need to check for an arbitrary point, say, $t_0\in D$, and we may assume $\alpha_{i\bar{j}}=\delta_{i\bar{j}}$ at $t_0$. Therefore, formula (\ref{2nd form}) at $t_0$ becomes
\begin{equation}\label{2nd form'}
    \sum_j (\Theta^V_{j\Bar{j}}s,s)=\sum_j (\phi_{j\Bar{j}}s,s)-\sum_j \|\pi_{\perp} D^F_{t_j}s\|^2.
\end{equation}
The key observation of Berndtsson is that the second fundamental form can be controlled using $L^2$-estimates, so we deduce from (\ref{2nd form'}) that
\begin{equation}\label{L2} 
\sum_j (\Theta^V_{j\Bar{j}}s,s)\geq \int_{X} K_{t_0}(\cdot)  h(s,s)e^{-\phi(t_0,\cdot)}
\end{equation}
where $K_{t_0}: X \to \mathbb{R}$ is a smooth function, given in local coordinates on $X$ by $$K_{t_0}=\sum_j(\phi_{j\Bar{j}}-\sum_{\lambda, \mu}(\psi+\phi)^{\Bar{\lambda} \mu}\phi_{j\Bar{\lambda}}\phi_{\Bar{j}\mu});$$ here $i\partial\bar{\partial}\psi=\omega$ and $(\psi+\phi)^{\Bar{\lambda} \mu}$ stands for the inverse matrix of $(\psi+\phi)_{\Bar{\lambda} \mu}$. (compare with \cite[Formula (3.1)]{Berndtsson09}).  

We claim that $K_{t_0}> 0$. In fact, we have 
$$\phi_{j\Bar{j}}-\sum_{\lambda, \mu}(\psi+\phi)^{\Bar{\lambda} \mu}\phi_{j\Bar{\lambda}}\phi_{\Bar{j}\mu}=\frac{\det M(j\bar{j})}{\det ((\phi+\psi)_{\mu\Bar{\lambda}})}$$
by using Schur's formula on the  matrix $M(j\bar{j})$ from (\ref{matrix}) and that $\psi$ is independent of $t$. Therefore, $K_{t_0}=\sum_j \det M(j\bar{j})/\det ((\phi+\psi)_{\mu\Bar{\lambda}})$. Using the assumption $(\pi_2^*\omega+i\partial \bar{\partial}\phi)^{n+1}\wedge(\pi_1^*\alpha)^{m-1} > 0$ and formula (\ref{2.1}), we see $K_{t_0}> 0$. As a result, (\ref{L2}) implies  $\sum_j(\Theta^V_{j\Bar{j}}s,s)> 0$, and so $\Lambda_\alpha \Theta^V>0$. For the semipositive case $(\pi_2^*\omega+i\partial \bar{\partial}\phi)^{n+1}\wedge(\pi_1^*\alpha)^{m-1} \geq 0$, we would get $K_{t_0}\geq 0$  and so $\Lambda_\alpha \Theta^V\geq  0$.         
\end{proof}

\subsection{Subharmonicity on graphs}\label{subsection 2.1}
In this subsection, we are going to generalize Theorem \ref{thm 1} to weight functions $\phi$ that are subharmonic on graphs.  
\begin{definition}\label{sub graphs}
 An upper semicontinuous function $\phi:D\times X\to [-\infty, \infty)$ is said to be $(\omega,\alpha)$-subharmonic on graphs if, for any holomorphic map $f$ from an open subset of $D$ to $X$, $\psi(f(t))+\phi(t,f(t))$ is $\alpha$-subharmonic, where $\psi$ is a local potential of $\omega$.
\end{definition}

When $\alpha$ is the Euclidean metric on $D\subset \mathbb{C}^m$, the definition is the one given in \cite{wu23}. For functions smooth enough, we have a characterization which explains why subharmonicity on graphs is a broader concept.

\begin{lemma}\label{+ det}
  Suppose $\phi$ is a $C^2$ function on $D\times X$ and $\omega +i\partial\Bar{\partial}\phi(t,\cdot)>0$ on $X$ for all $t\in D$.  Then $\phi$ is $(\omega,\alpha)$-subharmonic on graphs if and only if $$(\pi_2^*\omega+i\partial \bar{\partial}\phi)^{n+1}\wedge(\pi_1^*\alpha)^{m-1} \geq 0.$$
\end{lemma}
\begin{proof}
    It is a consequence of Lemma \ref{min lem} and the formula (\ref{2.1}).
\end{proof}

\begin{lemma}\label{min lem}
Assume that $\phi$ is a $C^2$ function on $D\times X$ and $\omega +i\partial\Bar{\partial}\phi(t,\cdot)>0$ on $X$ for all $t\in D$. Let $f$ be any holomorphic function from an open subset of $D$ to $X$. We have
\begin{equation}\label{17}
\Delta_\alpha(\psi(f(t))+\phi(t,f(t)))\geq \sum_{i,j}\alpha^{i\bar{j}}\frac{\det M(i\bar{j})}{\det (\psi_{\mu\Bar{\lambda}}+\phi_{\mu\Bar{\lambda}})} 
\end{equation}
where $\psi$ is a local potential of $\omega$, and $\phi_{\mu \bar{\lambda}}=\partial^2 \phi/\partial z_\mu \partial \bar{z}_\lambda$ and  $\psi_{\mu \bar{\lambda}}=\partial^2 \psi/\partial z_\mu \partial \bar{z}_\lambda$ with $z_\mu, z_\lambda$ local coordinates in $X$.

Moreover, fixing $(t_0,z_0)\in D\times X$, we can find $f$ with $f(t_0)=z_0$ such that the equality holds at $(t_0,z_0)$,
\begin{equation}
\Delta_\alpha(\psi(f(t))+\phi(t,f(t)))|_{t_0}= \sum_{i,j}\alpha^{i\bar{j}}\frac{\det M(i\bar{j})}{\det (\psi_{\mu\Bar{\lambda}}+\phi_{\mu\Bar{\lambda}})}\big|_{(t_0,z_0)}. 
\end{equation}

\end{lemma}
\begin{proof}[Proof of Lemma \ref{min lem}]
    This is almost identical to \cite[Lemma 11]{wu2024potential}. The only difference is that we change the Euclidean metric to $\alpha$, so we skip the proof.
\end{proof}

The generalization of Theorem \ref{thm 1} is the following.

\begin{theorem}\label{thm less regular}
    If $\phi$ is bounded and $(\omega,\alpha)$-subharmonic on graphs in $D\times X$, then the dual metric $H^*$ on $V^*$ is an $\alpha$-subharmonic norm function.
\end{theorem}
That $H^*$ is an $\alpha$-subharmonic norm function means $\log H^*(s,s)$ is $\alpha$-subharmonic for any local holomorphic section $s$ of $V^*$. Theorem \ref{thm less regular} has appeared in \cite{wu23} when $\alpha$ is the Euclidean metric. The proof begins with an approximation lemma (Lemma \ref{approx}) and then uses the smooth case, namely Theorem \ref{thm 1}. 
The proof of Lemma \ref{approx} is almost the same as that of \cite[Lemma 2.2]{wu23}, and we skip its proof. 

\begin{lemma}\label{approx}
 Let $\phi$ be bounded and $(\omega,\alpha)$-subharmonic on graphs in $D\times X$. Then for $D'$ relatively compact open in $D$, there exist $\varepsilon_j \searrow 0$ and $\phi_j\in C^\infty(D'\times X)$ decreasing to $\phi$, and $\phi_j$ is $((1-\varepsilon_j)\omega,\alpha)$-subharmonic on graphs. Namely, for any holomorphic map $f$ from an open subset of $D'$ to $X$, $\Delta_\alpha(\psi(f(t))+\phi_j(t,f(t)))\geq \varepsilon_j\Delta_\alpha(\psi(f(t))$, where $\omega=i\partial \Bar{\partial}\psi $ locally.
\end{lemma}

\begin{proof}[Proof of Theorem \ref{thm less regular}]
Being an $\alpha$-subharmonic norm function is a local property, so we focus on $D'$ a relatively compact open set in $D$. Take $\varepsilon_j$ and $\phi_j$ as in Lemma \ref{approx}. We claim that $\omega+i\partial\bar{\partial}\phi_j(t,\cdot)>0$ for any $t\in D'$. This can be seen by using \cite[Lemma 3.1]{wu23} which says that $\phi_j(t,\cdot)$ is $(1-\varepsilon_j)\omega$-psh on $X$. Alternatively, one can use the computation in \cite[Page 346, Line 5]{wu23}. 

Meanwhile, since $\varepsilon_j\Delta_\alpha(\psi(f(t))\geq 0$, the functions $\phi_j$ are $(\omega,\alpha)$-subharmonic on graphs, and hence $(\pi_2^*\omega+i\partial \bar{\partial}\phi_j)^{n+1}\wedge(\pi_1^*\alpha)^{m-1} \geq 0$ by Lemma \ref{+ det}. Using Theorem \ref{thm 1} for $\phi_j$, we see that the Hermitian bundles $(V, H_j)$ with the weight $\phi_j$ have semipositive mean curvature. Hence the dual metrics $H^*_j$ have seminegative mean curvature. By \cite[Theorem 4.1]{CofimanSemmes}, this implies $H^*_j$ are $\alpha$-subharmonic norm functions. Since $\phi_j$ decreases to $\phi$, the metric $H^*_j$ decreases to $H^*$ as well. It follows that $H^*$ is an $\alpha$-subharmonic norm function

\end{proof}

\section{Proof of Theorem \ref{thm 4}}\label{sec 3}

Recall the setup: $p:\mathcal{X}^{m+n}\to Y^m$ is the fibration, $(L,h)\to \mathcal{X}$ is the Hermitian line bundle, and $V\to Y$ is a holomorphic vector bundle with fibers $V_t=H^0(\mathcal{X}_t, L|_{\mathcal{X}_t}\otimes K_{\mathcal{X}_t}).$  The bundle $V$ carries a Hermitian metric $H(u,u)=\int_{\mathcal{X}_t}h(u,u)$. Assuming the curvature $\Theta$ of $h$ is positive on every fiber and $\Theta^{n+1}\wedge p^*\alpha^{m-1}> 0$, we want to show that the Hermitian metric $H$ has positive mean curvature $\Lambda_\alpha \Theta^V > 0$.

We will follow closely the computations of Berndtsson in \cite[Section 4]{Berndtsson09} and \cite[Section 2]{BoMathz}. A smooth local section $u$ of the bundle $V$ is represented by a smooth $(n,0)$-from with values in $L$ over $p^{-1}(W)$ for some open set $W$ in $Y$ such that the restriction to each fiber is holomorphic. Any representative of $u$ is denoted by $\mathbf{u}$. We use $(t_1,\dots, t_m)$ for local coordinates in $Y$. In general, we have $$\bar{\partial}\mathbf{u}=\sum_j d\bar{t}_j\wedge \nu_j+  
 \sum_j \eta_j\wedge dt_j$$ where $\eta_j$ are of bidegree $(n-1,1)$ and $\nu_j$ are of bidegree $(n,0)$ whose restrictions to fibers are holomorphic (thus $\nu_j$ define sections of the bundle $V$). The Hermitian holomorphic vector bundle $(V, H)$ admits the Chern connection $D=D'+D{''}$. The $(0,1)$-part $D^{''}$ is given by $D^{''}u=\sum_j \nu_j d\bar{t}_j$. Therefore, a section $u$ of $V$ is holomorphic if and only if $\bar{\partial}\mathbf{u}=  \sum \eta_j\wedge dt_j$. 
 
 For the $(1,0)$-part $D'$, we consider some local frame $e$ of $L$ and write $\mathbf{u}=u'\otimes e$ with $u'$ an $(n,0)$-form. Denote $h(e,e)=e^{-\phi}$ and define 
 \begin{equation}\label{phi}
     \partial^\phi \mathbf{u}:=(\partial^\phi u')\otimes e = e^\phi \partial(e^{-\phi}{u'})\otimes e 
 \end{equation}
which is an $(n+1,0)$-form with values in $L$. It is straightforward to check that (\ref{phi}) is independent of the choice of the frame $e$. We can write  $\partial^\phi \mathbf{u}=\sum_j dt_j\wedge \mu_j$ where $\mu_j$ are of bidegree $(n,0)$. If we denote by $P(\mu_j)$ the orthogonal projection of $\mu_j$ on the space of holomorphic forms on each fiber, then $D'u=\sum_j P(\mu_j)dt_j$.

Next, we are going to prove $\Lambda_\alpha\Theta^V>0$. Let $u$ be a local holomorphic section of $V$ such that $D'u=0$ at a point $t_0$ in $Y$ and $u(t_0)\neq 0$. We fix a coordinate system $(t_1,\dots, t_m)$ around $t_0$ such that the Hermitian metric $\alpha=i\sum_{j} dt_j\wedge d\bar{t}_j$ at $t_0$. A standard computation gives
\begin{equation}\label{standard}
 \Lambda_\alpha  \partial \bar{\partial} H(u,u)=-\sum_{j} H(\Theta^V_{j\bar{j}}u,u) \text{ at } t_0 
\end{equation}
if we denote the curvature of the bundle $V$ by $\Theta^V=\sum_{i,j} \Theta^V_{i\bar{j}}dt_i\wedge d\bar{t}_j$.

On the other hand, if we let $\mathbf{u}$ be a representative of $u$ and write $\mathbf{u}=u'\otimes e$ with $u'$ an $(n,0)$-form and $e$ some local frame of $L$, then 
\begin{equation}
 H(u,u)=p_*(c_n u'\wedge \overline{u'} e^{-\phi})   
\end{equation}
where $e^{-\phi}=h(e,e)$ and $c_n=i^{n^2}$. We have 
$$\bar{\partial}H(u,u)=p_*\big( c_n \bar{\partial} u'\wedge \overline{u'} e^{-\phi} +(-1)^n  c_n u'\wedge \bar{\partial}(\overline{u'}e^{-\phi}) \big). $$
Since $u$ is a holomorphic section of $V$, we have $\bar{\partial}\mathbf{u}=  \sum_j \eta_j\wedge dt_j$. So, the form $p_*( c_n \bar{\partial }u'\wedge \overline{u'} e^{-\phi})$ contains $dt_j$. But $p_*( c_n \bar{\partial }u'\wedge \overline{u'} e^{-\phi})$ as the push forward of an $(n,n+1)$-form is a $(0,1)$-form, so it must be 0. Hence, 
$$\bar{\partial}H(u,u)=p_*\big((-1)^n  c_n u'\wedge \bar{\partial}(\overline{u'}e^{-\phi}) \big)= p_*\big((-1)^n  c_n u'\wedge \overline{\partial^{\phi}u'} e^{-\phi} \big)$$
where the second equality is by $\partial^{\phi}u'=e^{\phi}\partial (e^{-\phi}u')$. After applying $\partial$, we see 
\begin{equation}\label{4.1}
   \partial\bar{\partial}H(u,u)= p_*\big( (-1)^n c_n \partial^{\phi}u'\wedge \overline{\partial^{\phi}u'}e^{-\phi}+c_n u'e^{-\phi}\wedge \partial\overline{\partial^{\phi}u'}
    \big).
\end{equation}
Since $\bar{\partial}\partial^\phi u'+\partial^\phi\bar{\partial}u'=\partial\bar{\partial}\phi\wedge u'$, the last term in (\ref{4.1}) is $c_n p_* \big( u'\wedge (\bar{\partial}\partial \phi \wedge \overline{u'}-\overline{\partial^\phi \bar{\partial}u'})   e^{-\phi}\big) $. Moreover, the form $p_*(u'\wedge \overline{\bar{\partial}u'}e^{-\phi})$ is 0 since it contains $d\bar{t}_j$ and, as the push forward of an $(n+1,n)$-form, it is a $(1,0)$-form. Therefore, \begin{equation}
    0=\bar{\partial}p_*(u'\wedge \overline{\bar{\partial}u'}e^{-\phi})=p_*\big(
\bar{\partial}u'\wedge \overline{\bar{\partial}u'}e^{-\phi}+(-1)^n u'\wedge \overline{\partial^\phi \bar{\partial}u' }e^{-\phi}
    \big).
\end{equation}
Formula (\ref{4.1}) becomes
\begin{equation}\label{4.3}
 \begin{aligned}   &\partial\bar{\partial}H(u,u)\\
 = &(-1)^n c_n p_*(   \partial^{\phi}u'\wedge \overline{\partial^{\phi}u'}e^{-\phi})
    +
    c_n p_* ( u'\wedge \bar{\partial}\partial \phi \wedge \overline{u'}e^{-\phi})
    +
    (-1)^n c_np_*  (\bar{\partial}u'\wedge \overline{\bar{\partial}u'}e^{-\phi}). 
    \end{aligned}
\end{equation}

According to \cite[Proposition 4.2]{Berndtsson09}, we can choose a representative $\mathbf{u}$ such that in $\bar{\partial}\mathbf{u}=  \sum \eta_j\wedge dt_j$, the $\eta_j$ is primitive on $\mathcal{X}_{t_0}$. Moreover, 
$\partial^\phi u'=0$ at $t_0$. (The statement in  \cite[Proposition 4.2]{Berndtsson09}
is $\partial^\phi u\wedge \widehat{dt_j}=0$ at $t_0$ for all $j$.
But in \cite[Page 549, Line 3]{Berndtsson09} one can actually see that   $\partial^\phi u'=\partial^\phi u-\sum\partial^\phi \gamma^kdt_k=\sum (\mu^k-\partial^\phi \gamma^k)\wedge dt_k=0$ at $t_0$). After using such representative, formula (\ref{4.3}) is reduced to 
\begin{equation}\label{4.4}
\partial\bar{\partial}H(u,u)
 = 
    -c_n p_* ( u'\wedge\overline{u'}\wedge \partial\bar{\partial} \phi  e^{-\phi})
    +
    (-1)^n c_np_*  (\bar{\partial}u'\wedge \overline{\bar{\partial}u'}e^{-\phi}) \text{ at }  t_0.    
\end{equation}
We take trace with respect to $\alpha$ on formula (\ref{4.4}) to get  
\begin{equation}\label{4.5'}
    \Lambda_\alpha\partial\bar{\partial}H(u,u)
 = 
    -c_n \Lambda_\alpha p_* ( u'\wedge\overline{u'}\wedge \partial\bar{\partial} \phi  e^{-\phi})
    +
    (-1)^n c_n \Lambda_\alpha p_*  (\bar{\partial}u'\wedge \overline{\bar{\partial}u'}e^{-\phi})\text{ at }  t_0. 
\end{equation}
Because $\bar{\partial}\mathbf{u}=  \sum \eta_j\wedge dt_j$ and $\mathbf{u}=u'\otimes e$ and if we write $\eta_j=\eta_j'\otimes e$, then $\bar{\partial}u'=\sum \eta_j'\wedge dt_j$. The Hermitian form $\alpha$ at $t_0$ is $i\sum dt_j\wedge d\bar{t}_j$, so the last term in (\ref{4.5'}) is equal to \begin{equation}\label{4.5}
   (-1)^n c_n\Lambda_\alpha\int_{\mathcal{X}_{t_0}}(-1)^{n}\sum \eta'_j\wedge 
\overline{\eta'}_k \wedge dt_j\wedge d\bar{t}_k e^{-\phi}=  c_n \int_{\mathcal{X}_{t_0}}\sum \eta'_j\wedge 
\overline{\eta'}_j  e^{-\phi}\leq 0;
\end{equation}
the last inequality is by the fact that the $\eta_j$ is primitive on $\mathcal{X}_{t_0}$.

 To determine the sign of the middle term in (\ref{4.5'}), we may consider $$ -c_n p_* ( u'\wedge\overline{u'}\wedge \partial\bar{\partial} \phi  e^{-\phi})\wedge \alpha^{m-1} = -c_n p_* ( u'\wedge\overline{u'}\wedge \partial\bar{\partial} \phi  e^{-\phi} \wedge p^* \alpha^{m-1}).$$
We claim that the $(n+m,n+m)$-form $$u'\wedge\overline{u'}\wedge \partial\bar{\partial} \phi  e^{-\phi} \wedge p^* \alpha^{m-1}$$ is semipositive at every point on $\mathcal{X}_{t_0}$ and positive at some point on $\mathcal{X}_{t_0}$.

We verify this claim at one point on the fiber $\mathcal{X}_{t_0}$ and we denote the point by $(t_0,z_0)$. We already have coordinates $ (t_1,\dots, t_m)$ around $t_0$ in $Y$ such that $\alpha=i\sum dt_j\wedge d\bar{t}_j$ at $t_0$. Moreover, according to the assumption in Theorem \ref{thm 4}, the curvature $i\partial \bar{\partial}\phi$ on the fiber $\mathcal{X}_{t_0}$ is positive, so we may choose coordinates $(z_1,\dots,z_n)$ around $z_0$ in $\mathcal{X}_{t_0} $  such that $\phi_{\lambda\bar{\mu}}=\delta_{\lambda\bar{\mu}}$ at $z_0$. We will use the local coordinates $(t_1,\dots, t_m,z_1,\dots,z_n)$ around $(t_0,z_0)$ in $\mathcal{X}$ for computation. The assumption in Theorem \ref{thm 4}, $(i\partial\bar{\partial}\phi)^{n+1}\wedge p^*\alpha^{m-1} >0, $ gives rise to 
\begin{equation}\label{4.7}
\sum_j  \big( \phi_{j\bar{j}}-\sum_\mu |\phi_{j\bar{\mu}}|^2         \big)>0 \text{ at } (t_0,z_0). 
\end{equation}
This can be seen by a variant of formula (\ref{2.1}) and Schur's formula.

For the $(n,0)$-form $u'$, we denote by $u_z dz$ the part in $u'$ with $dz_1\wedge\dots \wedge dz_n$, and by $u_{z_\lambda t_j}d\hat{z}_\lambda\wedge dt_j$ the part with $dz_1\wedge \dots \wedge dz_n\wedge dt_j$ omitting $dz_\lambda$. Then the $(n+m,n+m)$-form $u'\wedge\overline{u'}\wedge \partial\bar{\partial} \phi  e^{-\phi} \wedge p^* \alpha^{m- 1} $ equals 
\begin{align*}
e^{-\phi} \big(&\sum_{j,k} u_z dz\wedge \overline{u_zdz} \wedge\phi_{j\bar{k}} dt_j\wedge d\bar{t}_k \wedge p^*\alpha^{m-1}\\+&\sum_{\lambda,j,k}
u_z dz\wedge 
\overline{u_{z_\lambda t_k}d\hat{z}_\lambda\wedge dt_k}\wedge
\phi_{j\bar{\lambda}} dt_j\wedge d\bar{z}_\lambda \wedge p^*\alpha^{m-1}\\+&\sum_{\lambda,j,k}
u_{z_\lambda t_j}d\hat{z}_\lambda\wedge dt_j\wedge \overline{u_zdz}\wedge\phi_{\lambda \bar{k}} dz_\lambda \wedge d\bar{t}_k \wedge p^*\alpha^{m-1}
\\+
&\sum_{\lambda,\mu,j,k} u_{z_\lambda t_j}d\hat{z}_\lambda\wedge dt_j\wedge \overline{u_{z_\mu t_k}d\hat{z}_\mu\wedge dt_k}\wedge\phi_{\lambda \bar{\mu}} dz_\lambda\wedge d\bar{z}_\mu \wedge p^*\alpha^{m-1}\big)
\end{align*}
which can be organized as 
\begin{equation}\label{4.8}
\begin{aligned}
e^{-\phi}\big(\sum_j|u_z|^2\phi_{j\bar{j}}+&\sum_{\lambda, j} (-1)^{n-\lambda+1}u_z \overline{u_{z_\lambda t_j}} \phi_{j\bar{\lambda}}+\sum_{\lambda, j} (-1)^{n-\lambda+1}\overline{u_z} u_{z_\lambda t_j}\phi_{\lambda\bar{j}}\\+&\sum_{\lambda,\mu,j} (-1)^{\lambda+\mu}u_{z_\lambda t_j} \overline{u_{z_\mu t_j}}\phi_{\lambda \bar{\mu}} \big)
dz\wedge d\bar{z}\wedge \frac{p^*\alpha^m}{m}.
\end{aligned}    
\end{equation}
Using the fact $\phi_{\lambda\bar{\mu}}=\delta_{\lambda \bar{\mu}}$ at $z_0$ and completing the square, we see that formula (\ref{4.8}) equals
\begin{equation}\label{3.12}
e^{-\phi}\big(\sum_{j} (\phi_{j\bar{j}}-\sum_{ \lambda}|\phi_{\lambda\bar{j}}|^2)|u_z|^2+\sum_{j,\lambda} \big|\phi_{\bar{\lambda}j}u_z+(-1)^{n-\lambda+1}u_{z_\lambda t_j}\big|^2\big)dz\wedge d\bar{z}\wedge \frac{p^*\alpha^m}{m}    
\end{equation}
which is semipositive by (\ref{4.7}). Since $u(t_0)\neq 0$, $\mathbf{u}|_{\mathcal{X}_{t_0}} $ cannot be identically zero, hence $u_z$ is nonzero somewhere; if $u_z$ is nonzero at $(t_0,z_0)$, then (\ref{3.12}) is positive. Thus we prove the claim that the $(n+m,n+m)$-form $$u'\wedge\overline{u'}\wedge \partial\bar{\partial} \phi  e^{-\phi} \wedge p^* \alpha^{m-1}$$ is semipositive at every point on $\mathcal{X}_{t_0}$ and positive at some point on $\mathcal{X}_{t_0}$. 

As a result, the right hand side in (\ref{4.5'}) is negative, so $\Lambda_\alpha  \partial \bar{\partial} H(u,u)< 0$. By (\ref{standard}), we get   $\Lambda_\alpha\Theta^V> 0$. For the semipositive case $(i\partial\bar{\partial}\phi)^{n+1}\wedge p^*\alpha^{m-1} \geq 0 $, the $(n+m,n+m)$-form $u'\wedge\overline{u'}\wedge \partial\bar{\partial} \phi  e^{-\phi} \wedge p^* \alpha^{m-1}$ is semipositive, so $\Lambda_\alpha\Theta^V\geq  0$.

\subsection{Relative Bergman kernel metric}\label{sub 3.1}
The aim of the subsection is to prove Corollary \ref{cor relative}. Let us define the relative Bergman kernel metric $B$ on $L\otimes K_{\mathcal{X}/Y}$ following \cite{BerndtssonPaun}. Given a metric $h$ on $L$, we define the Hermitian metric $H$ on $V$  as in (\ref{metric}). Let $\{u_j\}$ be any choice of orthonormal basis for $V_t$ with respect to $H$. For a section $u$ of $L\otimes K_{\mathcal{X}/Y}$ over $\mathcal{X}_t$, we define 
\begin{equation}
    B(u,u):=\frac{u\otimes \bar{u}}{\sum_j u_j\otimes \bar{u}_j}.
\end{equation}
It is not hard to verify that the definition is independent of the choice of the orthonormal basis. Moreover, we have the extremal characterization 
\begin{equation}
 \sum_j B(u_j,u_j)=\sup_{s\in V_t, H(s,s)\leq 1} B(s,s).   
\end{equation}
If we denote by $dx/dt$ a local frame of $K_{\mathcal{X}/Y}$ and by $e$ a local frame of $L$, and write $u_j=f_j dx/dt\otimes e $  and $s=s'dx/dt\otimes e$, then 
\begin{equation}\label{extreme}
 \sum_j |f_j|^2=\sup_{s\in V_t, H(s,s)\leq 1} |s'|^2\, \text{ and }\, B(\frac{dx}{dt}\otimes e,\frac{dx}{dt}\otimes e)=\frac{1}{\sum_j |f_j|^2}.  
\end{equation}

Choose local coordinates $(t,z)$ around a point in $\mathcal{X}$ such that the map $p$ is the trivial fibration $(t,z)\to t$. These coordinates give local frames $p^*(dt)$, $dt\wedge dz$, and $dt\wedge dz /p^*(dt)$ for $p^*(K_Y)$, $K_\mathcal{X}$, and $K_{\mathcal{X}/Y}$ respectively. We write 
\begin{equation}\label{psi}
  B\big(\frac{dt\wedge dz}{p^*(dt)}\otimes e, \frac{dt\wedge dz}{p^*(dt)}\otimes e\big)=e^{-\psi(t,z)}  
\end{equation}
and claim that $\Delta_\alpha\psi(t, f(t))\geq 0$ for any holomorphic map $t\to f(t)$ under the assumption of Corollary \ref{cor relative}.   

We fix a holomorphic map $t\to f(t)$ and define a local section $\xi$ of $V^*$ as follows (using the same coordinates). For $s\in V_t$, we write $s=s' dt\wedge dz /p^*(dt)\otimes e$, and define $\xi(s)=s'(t,f(t))$. The local section $\xi$ is holomorphic because if $s$ is a holomorphic section of $V$, then $\xi(s)$ is holomorphic. If we denote the dual metric on $V^*$ by $H^*$, then \begin{equation}
  \|\xi\|_{H^*}^2=\sup_{s\in V_t, H(s,s)\leq 1} |\xi(s)|^2= \sup_{s\in V_t, H(s,s)\leq 1} |s'(t,f(t))|^2= e^{\psi(t,f(t))} 
\end{equation}
where the last equality is by (\ref{extreme}) and (\ref{psi}). According to Theorem \ref{thm 4}, the dual metric $H^*$ on $V^*$ has seminegative mean curvature, and hence the logarithmic length of a holomorphic section $\log \|\xi\|_{H^*}^2$ is $\alpha$-subharmonic
by \cite[Theorem 4.1]{CofimanSemmes}, so $\Delta_\alpha\psi(t,f(t))\geq 0$.

We have shown that $\psi(t,z)$ is subharmonic on graphs. By \cite[Lemma 3.1]{wu23}, we see $\psi(t,\cdot)$ is plurisubharmonic for fixed $t$, so $\psi(t,z )+\varepsilon \|z\|^2$ is strongly plurisubharmonic for fixed $t$ and subharmonic on graphs. Using a variant of Lemma \ref{+ det}, we get 
$$\big(i\partial \bar{\partial }(\psi(t,z)+\varepsilon \|z\|^2)\big)^{n+1}\wedge p^*\alpha^{m-1}\geq 0.$$ After pushing $\varepsilon$ to 0, we have $\Theta(B)^{n+1}\wedge p^*\alpha^{m-1}\geq 0$.

\section{The subharmonic analogue of the Griffiths conjecture}\label{The subharmonic analogue of the Griffiths conjecture}

In this section, we prove the subharmonic analogue of the Griffiths conjecture. Namely, the following three statements are equivalent.  \begin{enumerate}[label=\Alph*.]

\item $E$ admits a Hermitian metric with positive mean curvature.

         \item $O_{P(E^*)}(1)$ admits a Hermitian metric whose curvature $\Theta$ is positive on every fiber and $\Theta^{r}\wedge p^*\alpha^{n-1}> 0$.

      \item
      $O_{P(E)}(1)$ admits a Hermitian metric whose curvature $\Theta$ is positive on every fiber and $\Theta^{r}\wedge q^*\alpha^{n-1}<  0$, where $q:P(E)\to X$ is the projection.
      \end{enumerate}

We begin with a standard computation in Hermitian bundles. A Hermitian metric $H^*$ on $E^*$ induces a metric $h$ on $O_{P(E^*)}(1)$. Fix a point $t_0\in X$ and let $(t_1,\dots, t_n)$ be local coordinates around $t_0$ in $X$. Let $\{e^*_1,\dots,e^*_r\}$ be a holomorphic frame of $E^*$ normal at $t_0$. Let $(\zeta_1,\dots,\zeta_r)$ be local fiber coordinates of $E^*$ with respect to the frame $\{e_1^*,\dots, e_r^*\}$. 

For a point $(t_0,[\zeta_0])$ in $P(E^*)$, we assume the local coordinates $(t_1,\dots,t_n, w_2,\dots, w_{r})$ are given by $w_i=\zeta_i/\zeta_{1}$ for $2\leq i\leq r$, and the point $(t_0,[\zeta_0])$ corresponds to the origin in the local coordinates. Then $\epsilon(t,w):=e_1^*+w_2e_2^*+\cdots+w_re^*_r$ is a local frame for $O_{P(E^*)}(-1)$. The curvature $\Theta(h^*)=-i\partial\bar{\partial}\log h^*(\epsilon,\epsilon)$ at $(t_0,[\zeta_0])$ after a straightforward computation using normality is equal to 
\begin{equation}\label{777}
   H^*(\Theta^{E^*}e_1^*,e_1^*)-\sum_{\lambda}dw_\lambda\wedge d\bar{w}_\lambda.
\end{equation}
If we denote the curvature of $(E^*,H^*)$ by $\Theta^{E^*}=\sum \Theta^{E^*}_{j\bar{k}}dt_j\wedge d\bar{t}_k$, then at $(t_0,[\zeta_0])$ using (\ref{777}) we get
\begin{equation*}
    \Theta(h^*)^{r}\wedge p^*\alpha^{n-1}=r!(n-1)!\sum_{j,k} \alpha^{j\bar{k}}H^*(\Theta^{E^*}_{j\bar{k}}e^*_1, e^*_1)(-1)^{r-1}\det(\alpha)
  \big(\bigwedge^n_{k=1} i dt_k\wedge d\Bar{t}_k\wedge \bigwedge^{r}_{\lambda=2} i dw_\lambda\wedge d\Bar{w}_\lambda\big). 
\end{equation*}
After using $\Theta(h^*)=-\Theta(h)$ and getting rid of $(-1)^{r-1}$, we get
\begin{equation}\label{555}
    \Theta(h)^{r}\wedge p^*\alpha^{n-1}=-r!(n-1)!\sum_{j,k} \alpha^{j\bar{k}}H^*(\Theta^{E^*}_{j\bar{k}}e^*_1, e^*_1)
  \det (\alpha)\big(\bigwedge^n_{k=1} i dt_k\wedge d\Bar{t}_k\wedge \bigwedge^{r}_{\lambda=2} i dw_\lambda\wedge d\Bar{w}_\lambda\big). 
\end{equation}
Using formula (\ref{555}), we see that statement A implies both B and C. 

\subsection{Special case}\label{subsec special}

Next, let us prove that statement B implies statement A for some special cases. The first one is for semistable bundles. 
\begin{lemma}\label{semistable}
    If the bundle $E$ is semistable and satisfies statement B, then the bundle $E$ satisfies statement A.
\end{lemma}
\begin{proof}

By \cite{AdamJacob}, semistability is equivalent to the existence of approximate Hermitian--Yang--Mills metrics; that is, given $\varepsilon>0$, there exists a Hermitian metric $H_\varepsilon$ on $E$ such that $\max_{X}|\Lambda_\alpha \Theta(H_\varepsilon)-cI|<\varepsilon$ where $I$ is the identity endomorphism on $E$ and $$c=\frac{2n\pi\int_X c_1(E)\wedge \alpha^{n-1}}{r\int_X\alpha^{n}}.$$ 
Therefore, the remaining task is to show that $c>0$. This can be done using Theorem \ref{thm EdetE}. Indeed, the bundle $ \det E$ admits a Hermitian metric $H$ with $\Lambda_\alpha \Theta(H)>0$ by the first statement in Theorem \ref{thm EdetE} with $k=0$. We see
\begin{equation}\label{degree}    \int_Xc_1(E)\wedge \alpha^{n-1}= \int_X \frac{i}{2\pi}\Theta( H)\wedge \alpha^{n-1}=\int_X \frac{i}{2\pi} \Lambda_\alpha \Theta (H)\frac{\alpha^n}{n}
\end{equation}
which is positive by $\Lambda_\alpha \Theta(H)>0$. As a consequence, the constant $c$ is positive.
\end{proof}

\begin{lemma}\label{lemma quotient}
    If $E$ satisfies statement B, then any quotient bundle of $E$ also satisfies statement B and hence has positive degree.
\end{lemma}

\begin{proof}
Let $Q$ be a quotient bundle of $E$. Since $E$ satisfies statement $B$, there is a metric $h$ on $O_{P(E^*)}(1)$ whose curvature $\Theta$ is positive one every fiber and $\Theta^r \wedge p^*\alpha^{n-1}>0$. Since $Q^*$ is a subbundle of $E^*$, the inclusion maps $P(Q_t^*)\to P(E_t^*)$ for $t\in X$ piece together to give $\iota:P(Q^*)\to P(E^*)$. The pull-back bundle $\iota^*O_{P(E^*)}(-1)$ is just $O_{P(Q^*)}(-1)$, and so there is the induced metric $\iota^*h$ on $O_{P(Q^*)}(1)$. The curvature of the metric $\iota^*h$ when restricted to a fiber $P(Q^*_t)$ is $$\Theta(\iota^*h)|_{P(Q^*_t)}=\Theta|_{P(Q^*_t)}$$
which is positive because $\Theta|_{P(E^*_t)}$ is positive.

Denote by $p'$ the projection $P(Q^*)\to X$. We claim that $\Theta(\iota^*h)^{r'}\wedge (p')^*\alpha^{n-1}>0$ where $r'=\rank Q$. We verify the claim at one point $(t_0, [\zeta_0])\in P(Q^*)$. We fix a coordinate system $(t_1,\dots, t_n)$ around $t_0$ in $X$ such that the K\"ahler metric $\alpha=i\sum_{j} dt_j\wedge d\bar{t}_j$ at $t_0$.

For the point $(t_0, [\zeta_0])\in P(Q^*)$, there exists a holomorphic frame $\{e^*_1,\dots, e^*_{r'}\}$ of $Q^*$ around $t_0$ such that it is normal in the following sense. Let $(\zeta_1,\dots,\zeta_{r'})$ be local fiber coordinates of $Q^*$ with respect to the frame $\{e_1^*,\dots, e_{r'}^*\}$. We assume the local coordinates  around $(t_0, [\zeta_0])\in P(Q^*)$ are given by $(t_1,\dots,t_n, w_2,\dots, w_{r'})$ with $w_i=\zeta_i/\zeta_{1}$ for $2\leq i\leq r'$. Then $\epsilon':=e_1^*+w_2e_2^*+\cdots+w_{r'}e^*_{r'}$ is a local frame for $O_{P(Q^*)}(-1)$. Denote $\iota^*h^*(\epsilon',\epsilon')=e^{\phi'}$. Normality means 
\begin{equation}\label{normal}
    \phi'_{{t_j}\bar{w}_\mu}=0 \text{ for } 1\leq j\leq n \text{ and } 2\leq \mu \leq r' \text{ at  $(t_0,[\zeta_0])$}.
\end{equation}
The existence of such a frame $\{e^*_j\}$ of $Q^*$ will be proved in Subsection \ref{subsection 5.1} in the appendix (it is the normal frame for strongly pseudoconvex Finsler metrics).   

We extend the frame $\{e^*_1,\dots, e^*_{r'}\}$ of $Q^*$ to a frame $\{e^*_1,\dots, e^*_r\}$ of $E^*$ (This is doable because we can always find a frame $\{s^*_1,\dots, s^*_r\}$ of $E^*$ such that $\{s^*_1,\dots, s^*_{r'}\}$ is a frame for $Q^*$, and the desired extension is $\{e^*_1,\dots, e^*_{r'},s^*_{r'+1},\dots, s^*_r\}$). Let $(\zeta_1,\dots,\zeta_r)$ be local fiber coordinates of $E^*$ with respect to the frame $\{e_1^*,\dots, e_r^*\}$. For the point $(t_0, [\zeta_0])\in P(Q^*)\subset P(E^*)$, the local coordinates in $P(E^*)$ are given by $(t_1,\dots,t_n, w_2,\dots, w_{r})$  with $w_i=\zeta_i/\zeta_{1}$ for $2\leq i\leq r$. As before, $\epsilon:=e_1^*+w_2e_2^*+\cdots+w_{r}e^*_{r}$ is a local frame for $O_{P(E^*)}(-1)$. Denote $h^*(\epsilon,\epsilon)=e^{\phi}$. Therefore, the map $\iota(t_1,\dots,t_n,w_2,\dots, w_{r'})=(t_1,\dots,t_n,w_2,\dots, w_{r'},0,\dots,0)$, $\iota^*\epsilon=\epsilon'$, and $\iota^*\phi=\phi'$.

Since
$\Theta^{r}\wedge p^*\alpha^{n-1}>0$, by a variant of formula (\ref{2.1}) and Schur's formula, we see 
\begin{equation}
\sum_j \big(\phi_{j\bar{j}}-\sum_{2\leq \lambda, \mu\leq r}\phi_{j\bar{\mu}}\phi^{\lambda\bar{\mu}}\phi_{\lambda\bar{j}}
\big)\det(\phi_{\lambda\bar{\mu}})>0 \text{ at } (t_0,[\zeta_0]),    
\end{equation}
and so 
\begin{equation}\label{111}
 \sum_j \phi_{j\bar{j}}>\sum_j\sum_{2\leq \lambda, \mu\leq r}\phi_{j\bar{\mu}}\phi^{\lambda\bar{\mu}}\phi_{\lambda\bar{j}}\geq 0 \text{ at } (t_0,[\zeta_0])  
\end{equation}
where the last inequality is due to positivity of the matrix $(\phi_{\lambda\bar{\mu}})$. On the other hand, using a variant of formula (\ref{2.1}) and normality (\ref{normal}), 
we have at $(t_0,[\zeta_0])$
\begin{equation*}
\Theta(\iota^*h)^{r'}\wedge (p')^*\alpha^{n-1}=r'!(n-1)!\sum_{j} \phi'_{j\bar{j}}\det\big[(\phi'_{\lambda\bar{\mu}})_{2\leq \lambda, \mu\leq r'}\big]  \big(\bigwedge^n_{k=1} i dt_k\wedge d\Bar{t}_k\wedge \bigwedge^{r'}_{\lambda=2} i dw_\lambda\wedge d\Bar{w}_\lambda\big)
\end{equation*}
which is positive by (\ref{111}) and the fact $\iota^*\phi=\phi'$. Hence we prove the claim that $\Theta(\iota^*h)^{r'}\wedge (p')^*\alpha^{n-1}>0$ (another way to prove the claim is to use a variant of Lemma \ref{+ det} and the fact that if $\phi$ is subharmonic on graphs, then so is $\iota^*\phi$). 

All in all,  we see that the quotient bundle $Q$ satisfies statement B. In order to show the degree of $Q$ is positive, we use Theorem \ref{thm EdetE}. By Theorem \ref{thm EdetE}, 
the bundle $ \det Q$ admits a Hermitian metric $H$ with $\Lambda_\alpha \Theta(H)>0$, and hence the degree $\int_X c_1(Q)\wedge \alpha^{n-1}$ of $Q$ is positive by the same computation as in (\ref{degree}).    
\end{proof}

\begin{lemma}\label{short}
    In the short exact sequence of vector bundles $0\to E_1\to E\to E_2\to 0$, if $E_1$ and $E_2$ satisfy statement A, then so does $E$.
\end{lemma}
\begin{proof}
    The proof is similar to \cite[Lemma 2.2]{Umemura}. One simply changes Nakano positivity to mean curvature positivity.
\end{proof}

\begin{theorem}\label{thm b to a}
Assume $E$  is of rank 2 or the dimension of $X$ is 2. If the degree of any quotient bundle of $E$ (including $E$ itself) is positive, then  $E$ satisfies statement A.
\end{theorem}

\begin{proof}
We consider the case $\dim X=2$ first and will follow closely the argument in \cite[Proof of Theorem 2.6]{Umemura}. We use induction on the rank of $E$. For $\rank E=1$, since a line bundle always admits a Hermitian--Yang--Mills metric (\cite[Propositions 4.1.4 and 4.2.4]{reprintDiffofcomplexbundles}), if the degree of $E$ is positive, then $E$ satisfies statement A.

Next, we assume the theorem is true for bundles of rank less than $r$. Let $E$ be a bundle of rank $r$ and assume the degree of any quotient bundle of $E$ is positive. If $E$ has a subbundle $E_1$ satisfying statement A, then $E$ satisfies statement A; indeed, the quotient bundle $E_2$ in the sequence  $0\to E_1\to E\to E_2\to 0$ satisfies statement A by the induction hypothesis, so the bundle $E$ satisfies statement A by Lemma \ref{short}.

Assume $E$ does not have a subbundle satisfying statement A. We claim that $E$ is stable. It is enough to check the slope for reflexive subsheaf $S$ with $0<\rank S< \rank E$ (\cite[Propositions 5.5.22 and 5.7.6]{reprintDiffofcomplexbundles}). Now we use the assumption that the dimension of $X$ is 2 to deduce that reflexive sheaf is locally free (\cite[Corollary 5.5.20]{reprintDiffofcomplexbundles}), so we only need to consider subbundles $S$ and we claim that $\deg S\leq 0$. We use induction on the rank of $S$. For $\rank S=1$, we know $S$ admits Hermitian--Yang--Mills metric, so $\deg S \leq 0$, otherwise $E$ would have a subbundle satisfying statement A, a contradiction. Next we assume, for subbundles of rank less than $r'$, the degree is nonpositive. Let $S$ be a subbundle of rank $r'$. Suppose $\deg S>0$. By the induction hypothesis (the second one), $S$ is stable and hence admits a Hermitian--Yang--Mills metric by \cite{uhlenbeckyau} and satisfies statement A, but  this contradicts with the fact $E$ does not have a subbundle satisfying statement A. Therefore, $\deg S\leq 0$. (another way to get a contradiction is to notice that if $\deg S>0$, then any quotient bundle of $S$ would have positive degree by the second induction hypothesis, and $S$ would satisfy statement A by the first induction hypothesis). Therefore, all the subbundles of $E$ have nonpositive degree, so $E$ is stable because $\deg E>0$. By the existence of Hermitian--Yang--Mills metric on $E$ and $\deg E>0$, the bundle $E$ satisfies statement A. This completes the proof for $\dim X=2$.

For $\rank E =2$, the proof is similar. If $E$ has a subbundle $E_1$ satisfying statement A, then the quotient bundle $E_2$ in the sequence  $0\to E_1\to E\to E_2\to 0$ satisfies statement A since $\rank E_2=1$ and $\deg E_2>0$, so the bundle $E$ satisfies statement A by Lemma \ref{short}. Assume $E$ does not have a subbundle satisfying statement A. We claim that $E$ is stable. Again, it is enough to check the slope for reflexive subsheaf $S$ with $0<\rank S< \rank E$. Since $\rank E =2$, $S$ is of rank one and hence locally free (\cite[Lemma 1.1.15]{vectorbundlesonproj}). As a line bundle, $S$ admits a Hermitian--Yang--Mills metric, if $\deg S>0$, then $S$ satisfies statement A which contradicts with the fact $E$ does not have a subbundle satisfying statement A. Therefore, $\deg S \leq 0$, and $E$ is stable because $\deg E >0$. By the existence of Hermitian--Yang--Mills metric on $E$ and $\deg E>0$, the bundle $E$ satisfies statement A.
\end{proof}
Using Lemma \ref{lemma quotient} and Theorem \ref{thm b to a}, we deduce
\begin{corollary}\label{cor b to a}
    Assume $E$  is of rank 2 or the dimension of $X$ is 2.
    If the bundle $E$ satisfies statement B, then $E$ satisfies statement A.
\end{corollary}

\subsection{General case}\label{subsec general}
In this subsection, we are going to prove that statements B and C both imply A in full generality. The asymptotic expansion formulas in \cite{Mazhang} and the fibered Yang--Mills functionals in \cite{finski2024lower} play important roles in the proofs. We begin with a lemma.

\begin{lemma}\label{lem horizontal}
     Let $h$ be a Hermitian metric on the line bundle $O_{P(E^*)}(1)$ such that its curvature $\Theta$ is positive on every fiber. Define the horizontal mean curvature $$\wedge_{\alpha}\Theta_H:=\frac{\Theta_H\wedge \alpha^{n-1}}{\alpha^n},$$
    where $\Theta_H$ is the horizontal component of $\Theta$. The horizontal mean curvature $\wedge_{\alpha}\Theta_H
     $ is positive (negative) if and only if 
     $\Theta^{r}\wedge p^*\alpha^{n-1}> 0\, (<0) $.
\end{lemma}
\begin{proof}
    Since the curvature $\Theta$ is positive on every fiber, we can find normal coordinates around a fixed point $(t_0,[\zeta_0])$ in $P(E^*)$. Indeed, using the results and the notation from Subsection \ref{subsection 5.1}, we have at $(t_0,[\zeta_0])$
 $$\Theta=\sum\phi_{i\bar{j}}dt_i \wedge d\bar{t}_j+\sum \phi_{\lambda\bar{\mu}}dw_\lambda\wedge d\bar{w}_\mu$$
where $h^*(\epsilon,\epsilon)=e^{\phi}$. Therefore, at $(t_0,[\zeta_0])$, the horizontal component $\Theta_H=\sum\phi_{i\bar{j}}dt_i \wedge d\bar{t}_j$ and the horizontal mean curvature
$
\wedge_{\alpha}\Theta_H=\sum \alpha^{j\bar{k}} \phi_{j\bar{k}}/n    
$.
Meanwhile, according to a variant of formula (\ref{2.1}), we have at $(t_0,[\zeta_0])$
\begin{equation*}
\Theta^{r}\wedge p^*\alpha^{n-1}=r!(n-1)!\sum_{j,k} \alpha^{j\bar{k}} \phi_{j\bar{k}}\det(\phi_{\lambda\bar{\mu}})
  \det (\alpha)\big(\bigwedge^n_{k=1} i dt_k\wedge d\Bar{t}_k\wedge \bigwedge^{r}_{\lambda=2} i dw_\lambda\wedge d\Bar{w}_\lambda\big). 
\end{equation*}
So the lemma follows. 
\end{proof}

For the bundle of symmetric power $S^kE$, we consider its Harder--Narasimhan filtration and the slopes of the quotient sheaves in the filtration which are called the Harder--Narasimhan slopes. Let $\mu^k_{\min}$ and $\mu^k_{\max}$ be the minimal and maximal Harder--Narasimhan slopes.

\begin{lemma}\label{lem k times}
We have the relation $\mu^k_{\min}=k\mu^1_{\min}$ and $\mu^k_{\max}=k\mu^1_{\max}$.   
\end{lemma}

\begin{proof}
According to \cite[Propositions 3.2 and 3.4]{chen2011computing}, the relation is true for $\dim X=1$. One can then proceed by induction on $\dim X$ and use the Mehta--Ramanathan theorem  in \cite{MehtaRamanathan} to conclude the lemma. \end{proof}

\begin{theorem}
    If the bundle $E$ satisfies statement B, then $E$ satisfies statement A.
\end{theorem}
\begin{proof}
We provide two proofs. For the first proof, we see that the horizontal mean curvature is positive 
by Lemma \ref{lem horizontal}. According to \cite[Formula (1.10)]{finski2024lower}, the asymptotic slope $\eta^{HN}_{\min}:=\lim_{k\to \infty } \mu^k_{\min} /k $ is positive. Since $\mu^k_{\min}=k \mu^1_{\min}$ by Lemma \ref{lem k times}, we see $\eta^{HN}_{\min}= \mu^1_{\min} $ which is positive. Hence the bundle $E$ satisfies statement A by \cite[Theorem 1.4]{li2021mean} (see also \cite[Corollary 3.6]{finski2024lower}). 

For the second proof, 
we use Theorem \ref{thm EdetE} to deduce that the bundle $S^kE$ satisfies statment A for $k$ large. So, by \cite[Theorem 1.4]{li2021mean}, the minimal slope $\mu^k_{\min} $ is positive for $k$ large. Using Lemma \ref{lem k times}, we have $\mu^1_{\min} >0$ and so the bundle $E$ satisfies statement A by \cite[Theorem 1.4]{li2021mean} again.
\end{proof}

\begin{theorem}
    If the bundle $E$ satisfies statement C, then $E$ satisfies statement A.
\end{theorem}

\begin{proof}
By Lemma \ref{lem horizontal}, we see that the horizontal mean curvature $\wedge_\alpha \Theta_H$ is negative on $P(E)$. According to \cite[Formula (1.10)]{finski2024lower},  the asymptotic slope $\eta^{HN}_{\max}:=\lim_{k\to \infty } \mu^k_{\max} /k$ is negative for $E^*$. Since $\mu^k_{\max}=k \mu^1_{\max}$ by Lemma \ref{lem k times}, we see $\eta^{HN}_{\max}= \mu^1_{\max} $ which is negative for $E^*$. By \cite[Theorem 1.4]{li2021mean}, the bundle $E^*$ carries a Hermitian metric with negative mean curvature. Therefore, the bundle $E$ satisfies statement A .
\end{proof}

\section{Appendix}\label{sec 5}

\subsection{}\label{subsection 5.1}

Let $E$ be a holomorphic vector bundle of rank $r$ over a compact complex manifold $X$ of dimension $n$. We denote the dual bundle by $E^*$. For a vector $\zeta\in E^*_t$, we will symbolically write $(t,\zeta)\in E^*$. Denote by $P(E^*)$ the projectivized bundle of $E^*$, and by $O_{P(E^*)}(-1)$ the tautological line bundle over $P(E^*)$. Let $p$ be the projection from $P(E^*)$ to $X$. For a vector $\zeta\in E^*_t$, we denote by $[\zeta]$ the equivalence class of $\zeta$ in $P(E^*_t)$, and we will write $(t,[\zeta])\in P(E^*)$.

When we need to do local computations, we use $(t_1,...,t_n)$ for local coordinates on $X$, and $(\zeta_1,...,\zeta_r)$ for local fiber coordinates on $E^*$ with respect to a holomorphic frame $\{e_1^*,...,e_r^*\}$. So $(t_1,...,t_n,\zeta_1,...,\zeta_r)$ is a coordinate system on $E^*$. 

Let $G$ be a real-valued function on $E^*$ such that $G$ is smooth on $E^*\setminus \{0\}$ and $G(t,c \zeta)=|c|^2G(t,\zeta)$ for $c \in \mathbb{C}$. We write 
\begin{align*}
G_\lambda=\partial G/\partial \zeta_\lambda\,,\text{      }\text{  }G_{\bar{\mu}}=\partial G/\partial \bar{\zeta}_\mu\,,\text{      }\text{  } G_{\lambda\Bar{\mu}}=\partial^2 G/\partial \zeta_\lambda \partial\bar{\zeta}_{\mu}\,, \\    G_{\lambda i}=\partial G_\lambda/\partial t_i\,,\text{      }\text{  } G_{\lambda\bar{\mu}\bar{j}}=\partial G_{\lambda\bar{\mu}}/\partial \bar{t}_j\,,\text{ etc.,}
\end{align*}  
with Greek letters $\lambda,\mu$ for the fiber direction and Latin letters $i,j$ for the base direction.  The following lemma states the existence of normal coordinates. It is mentioned in \cite{Negfinsler,ComplexFinsler} without a proof. We provide a proof here (the notation we use is slightly different from \cite{Negfinsler,ComplexFinsler}). 
\begin{lemma}\label{normal finsler}
Assume the fiberwise complex Hessian of $G$ is positive. Given a point $(t_0,\zeta_0)\in E^*$, we can find a local holomorphic frame $\{e_1^*,\dots, e_r^*\}$ such that 
\begin{align*}
    G_{\lambda\bar{\mu}}(t_0,\zeta_0)=\delta_{\lambda\bar{\mu}},\,\, G_{\lambda\bar{\mu}i}(t_0,\zeta_0)=G_{\lambda\bar{\mu}\bar{j}}(t_0,\zeta_0)=0,\\
    G_{\lambda\bar{j}}(t_0,\zeta_0)=G_{\lambda j}(t_0,\zeta_0)=0, \,\, G_{\bar{j}}(t_0,\zeta_0)=0.
\end{align*}

\end{lemma}

\begin{proof}

We begin with an arbitrary frame $\{s^*_1,\dots, s^*_r\}$ around $t_0$, and we denote the fiber coordinates by $(\eta_1,\dots, \eta_r)$. Since the fiberwise complex Hessian of $G$ is positive, after applying to the frame $\{s^*_1,\dots, s^*_r\}$ a linear transformation of constant coefficients, we may assume $G_{\eta_\lambda \bar{\eta}_\mu}(t_0,\zeta_0)=\delta_{\lambda\mu}$. So, we have an expansion 
\begin{equation}
    G_{\eta_\lambda \bar{\eta}_\mu}=\delta_{\lambda\mu}+\sum_{i} a_{i \lambda\mu}t_i+\sum_{i} a'_{i \lambda\mu}\bar{t}_i+O(|t|^2)+O(|\eta|)+O(|\eta||t|),
\end{equation}
where $\overline{a_{i \lambda\mu}}=a'_{i \mu \lambda}$. We define a new frame $\{e^*_1,\dots, e^*_r\}$ by setting $e^*_\lambda=\sum_\mu A_{\lambda \mu}(t) s^*_\mu$ where $A_{\lambda \mu}(t)=\delta_{\lambda \mu}-\sum_i a_{i \lambda \mu } t_i$. We denote the fiber coordinates with respect to the frame $\{e^*_1,\dots, e^*_r\}$ by $(\zeta_1,\dots,\zeta_r)$. Then by a straightforward computation and the fact $O(|\eta|)=O(|\zeta|)+O(|t||\zeta|)$ we get
\begin{align*}
 G_{\zeta_\lambda \bar{\zeta}_\mu}=\sum_{\alpha,\beta}G_{\eta_\alpha\bar{\eta}_\beta}A_{\lambda\alpha}\bar{A}_{\mu\beta}= \delta_{\lambda\mu}+O(|t|^2)+O(|\zeta|)+O(|\zeta||t|)
  %=\sum_{m,n } \big(  \delta_{mn}+\sum_{\alpha} a_{\alpha mn}z_\alpha+\sum_{\alpha} a'_{\alpha mn}\bar{z}_\alpha+O(|z|^2)+O(|\zeta|)+O(|\zeta||z|)  \big)\big( \delta_{im}-\sum_{\alpha } a_{\alpha im} z_\alpha \big)\big( \delta_{jn}- \sum_{\beta } \overline{a_{\beta jn}} \bar{z}_\beta\big)
\end{align*}
which proves the first two sets of equalities in the lemma. For the last two sets of equalities, we take the homogeneity $G(t,c \zeta)=|c|^2G(t,\zeta)$ and differentiate with respect to $c$ to get 
\begin{equation}\label{first}
\sum_\mu G_\mu(t,c\zeta) \zeta_\mu=\bar{c}G(t,\zeta),    
\end{equation}
then differentiate with respect to $\bar{\zeta}_\lambda$
to get 
\begin{equation}\label{second}
    \sum_{\mu} G_{\mu\bar{\lambda}}(t,c\zeta)\zeta_\mu=G_{\bar{\lambda}}(t,\zeta).
\end{equation}
Finally, differentiating one more time with respect to $t_i $, we get $G_{\bar{\lambda}i}(t_0,\zeta_0)=0$ using $G_{\lambda\bar{\mu}i}(t_0,\zeta_0)=0$. That $G_{\lambda j}(t_0,\zeta_0)=0$ is proved similarly. Moreover, combining (\ref{first}) and (\ref{second}), we see $$\sum_{\mu,\lambda} G_{\mu\bar{\lambda}}(t,\zeta)\zeta_\mu \bar{\zeta}_\lambda=G(t,\zeta).$$
Differentiating with respect to $t_i$, we get $G_{i}(t_0,\zeta_0)=0$.
\end{proof}

Now, a metric $h^*$ on the line bundle  $O_{P(E^*)}(-1)$ defines a function $G$ on $E^*$ as follows  
$$G(t,\zeta):=h^*_{(t,[\zeta])}(\zeta,\zeta).$$
The function $G$ satisfies $G(t,c \zeta)=|c|^2G(t,\zeta)$ for $c \in \mathbb{C}$ (it is a Finsler metric). 

If the restriction of the curvature of $h^*$ to each fiber ${P(E_t^*)}$ is negative, then the fiberwise complex Hessian of $G$ is positive (see \cite[Lemma 3]{wu_2022}). Therefore, we can use Lemma \ref{normal finsler} to get a frame $\{e_1^*,\dots, e_r^*\}$ for a point $(t_0,\zeta_0)$ in $E^*$. Let $(\zeta_1,\dots,\zeta_r)$ be local fiber coordinates of $E^*$ with respect to the frame $\{e_1^*,\dots, e_r^*\}$. We assume the local coordinates around $(t_0, [\zeta_0])\in P(E^*)$ are given by $(t_1,\dots,t_n, w_2,\dots, w_{r})$  with $w_i=\zeta_i/\zeta_{1}$ for $2\leq i\leq r$. Then $\epsilon:=e_1^*+w_2e_2^*+\cdots+w_{r}e^*_{r}$ is a local frame for $O_{P(E^*)}(-1)$. Denote $h^*(\epsilon,\epsilon)=e^{\phi}$. So, $$\phi(t_1,\dots,t_n, w_2,\dots, w_{r})=\log G(t,\epsilon)=\log G(t, 1,w_2,\dots, w_r).$$
We simply compute $$\phi_{t_i \bar{w}_\mu}=\frac{G_{t_i \bar{\zeta}_\mu}G-G_{t_i}G_{\bar{\zeta}_\mu}}{G^2}$$
which is zero at $(t_0,\zeta_0)$ by Lemma \ref{normal finsler}. The above results are applied to $P(Q^*)$ with the metric $\iota^*h^*$ on $O_{P(Q^*)}(-1)$ in the proof of Lemma \ref{lemma quotient}.

\subsection{}\label{subsection 5.2}
Let $E$ be a holomorphic vector bundle of rank $r$ over a compact complex manifold $X$ of dimension $n$. Let $\alpha$ be a K\"ahler form on $X$ and $p$ be the projection from $P(E^*)$ to $X$. The goal is to show that $P(E^*)$ is K\"ahler.

Take an arbitrary Hermitian metric $H$ on $E^*$. The Hermitian metric $H$ induces a Hermitian metric $h$ on $O_{P(E^*)}(1)$. Fix a point $t_0\in X$ and let $(t_1,\dots, t_n)$ be local coordinates around $t_0$ in $X$. Let $\{e^*_1,\dots,e^*_r\}$ be a holomorphic frame of $E^*$ normal at $t_0$. Let $(\zeta_1,\dots,\zeta_r)$ be local fiber coordinates of $E^*$ with respect to the frame $(e_1^*,\dots, e_r^*)$. 

For a point $(t_0,[\zeta_0])$ in $P(E^*)$, we assume the local coordinates $(t_1,\dots,t_n, w_2,\dots, w_{r})$ are given by $w_i=\zeta_i/\zeta_{1}$ for $2\leq i\leq r$, and the point $(t_0,[\zeta_0])$ corresponds to the origin in the coordinates. By the same computation as in (\ref{777}), the curvature $\Theta(h)$ at $(t_0, [\zeta_0])$ is of the form 
$$\sum_{i,j} A_{i\bar{j}} dt_i\wedge d\Bar{t}_j+ \sum_{\lambda}  dw_\lambda \wedge d\Bar{w}_\lambda;$$ 
there is no mixed term. Meanwhile, if the pull back form $p^*\alpha$  is 
$$p^*\alpha=  \sum_{i,j} \alpha_{i\bar{j}} dt_i \wedge d\Bar{t}_j,$$ then
for a positive integer $k$, the form $kp^*\alpha + \Theta(h)$ at $(t_0, [\zeta_0]) \in P(E^*)$ is 
$$  \sum_{i,j}  \big(k \alpha_{i\bar{j}} +A_{i\bar{j}}\big) dt_i\wedge d\Bar{t}_j+ \sum_{\lambda} dw_\lambda \wedge d\Bar{w}_\lambda.    $$
Since $(\alpha_{i\bar{j}})$ is positive, we can make $\big(k\alpha_{i\bar{j}} +A_{i\bar{j}}\big)$ positive by taking $k>k_0$ for some integer $k_0$ depending on the point $(t_0, [\zeta_0])$.

Therefore, for $k>k_0$, the form $kp^*\alpha +\Theta(h)$ is positive at  $(t_0, [\zeta_0])\in P(E^*)$, hence in a neighborhood of this point. Since $P(E^*)$ is compact, by a compactness argument, we can find an integer $l_0$ such that  $kp^*\alpha +\Theta(h)$ is positive on $P(E^*)$ for $k>l_0$. So, we have a K\"ahler form $kp^*\alpha +\Theta(h)$ on $P(E^*)$.

\bibliographystyle{amsalpha}
\bibliography{Dominion}

\textsc{Erdős Center, HUN-REN Rényi Institute}

\texttt{\textbf{wuuuruuu@gmail.com}}

\end{document}